\documentclass[10]{amsart}
\usepackage{amsbsy,amssymb,pstricks,pst-node,mathrsfs,MnSymbol}
\usepackage{subfigure}
\usepackage[utf8]{inputenc}
\usepackage[left]{lineno}
\usepackage{tikz}
\usetikzlibrary{positioning}
\theoremstyle{definition}

\newtheorem{theorem}{Theorem}[section]
\newtheorem{thm}[theorem]{Theorem}
\newtheorem{lemma}[theorem]{Lemma}
\newtheorem{remark}[theorem]{Remark}

\newtheorem{corollary}[theorem]{Corollary}

\newtheorem{definition}[theorem]{Definition}

\renewcommand{\red}{\mathord{\mathrm{red}}}

\begin{document}
	
	%\internallinenumbers

	%
	\title{ Component graphs of vector spaces and zero-divisor graphs of ordered sets  }\maketitle
	\markboth{ Nilesh Khandekar,   Peter J. Cameron and Vinayak Joshi}{Component graphs of vector spaces and zero-divisor graphs of ordered sets   }\begin{center}\begin{large} Nilesh Khandekar$^1$,   Peter J. Cameron$^2$ and Vinayak Joshi$^1$ \end{large}\\\begin{small}\vskip.1in\emph{1. Department of Mathematics,
				Savitribai Phule Pune University,\\ Pune - 411007, Maharashtra,
				India\\
		2.  School of Mathematics and Statistics, University of St Andrews,\\ North Haugh, St Andrews,
		Fife KY16 9SS, UK}\\E-mail: \texttt{ khandekarnilesh11@gmail.com,   pjc20@st-andrews.ac.uk,  vvjoshi@unipune.ac.in, vinayakjoshi111@yahoo.com}\end{small}\end{center}\vskip.2in
	
	\begin{abstract} In this paper, nonzero component graphs and nonzero component union graphs of  finite dimensional vector space are studied using the zero-divisor graph of specially constructed  $0$-$1$-distributive lattice and the zero-divisor graph of rings. Further, we define an equivalence relation on nonzero component graphs and nonzero component union graphs to deduce that these graphs are the graph join of zero-divisor graphs of Boolean algebras and complete graphs. In the last section, we characterize the perfect and chordal nonzero component graphs and nonzero component union graphs.
	\end{abstract}\vskip.2in

	\noindent\begin{Small}\textbf{Mathematics Subject Classification (2020)}: 15A03; 05C17; 05C25; 06A07; 06A11.
		    \end{Small}\vskip.2in
	
	\vskip.25in

	\baselineskip 14truept
	\section{Introduction}\label{intro}
The study of graphs associated with algebraic and ordered structures is an active and fruitful area of research. Algebraic structures mainly include groups, rings and vector spaces while ordered structures include posets, lattices and boolean algebras. There are many research articles on associating a graph with an algebraic structure (ordered structure) and investigating the algebraic (ordered) features of algebraic (ordered) structure using the associated graph, and vice versa. Cayley graphs of groups is the first example of this kind. This graph was introduced by Arthur Cayley in 1878. Few other graphs such as zero-divisor graphs of rings, comaximal ideal graphs of rings, nonzero component graph (nonzero component union graph) of finite dimensional vector space  are some examples of the graphs associated with algebraic structures, while zero-divisor graphs of posets and comparability graphs of posets are some of the graphs associated with ordered structures.

Beck \cite{Be} first introduced the concept of a zero-divisor graph $\Gamma(R)$ of a commutative ring $R$ with unity, where the vertex set is the set of elements of $R$, and two vertices $x$ and $y$ are adjacent if $xy = 0$. Anderson and Livingston \cite{AL} modified this definition of  zero-divisor graph $\Gamma(R)$ of a ring $R$ by considering the vertex set to be the set of all nonzero zero-divisors and the adjacency to be the same, that is, $x$ and $y$ are adjacent if $xy = 0$.

The zero-divisor graph ${G}(P)$ of a poset $P$ is defined and explored in a similar way. In \cite{HJ}, the concept of a poset's zero-divisor graph is introduced, which is later modified in \cite{LW}. Assume that $P$ is a poset with $0$. Given any $\emptyset\neq A\subseteq P$,  the \textit{lower cone} of $A$ is given by $A^\ell=\{b\in P$ $|$
$b\leq a$ for every $a\in A\}$. Define a \emph{zero-divisor} of $P$ to be any element of the set
\[Z(P)=\left\{a\in P\mid(\exists b\in P\setminus\{0\}) \{a,b\}^\ell=\{0\}
\right\}.\]
As in \cite{LW}, the \emph{zero-divisor graph} of $P$ is the graph ${G}(P)$ whose vertices are the elements of $Z(P)\setminus\{0\}$ such that two vertices $a$ and $b$ are adjacent if and only if $\{a,b\}^\ell=\{0\}$.
If $Z(P)\neq\{0\}$, then clearly $G(P)$ has at least two vertices, and $G(P)$ is connected with diameter at most three \cite[Proposition 2.1]{LW}.
	
One good reason for studying the zero-divisor graphs of posets is that many 
other examples of graphs defined on algebraic structures, such as the 
\emph{noncyclic graph} of a group, defined by Abdollahi and Hassanabadi
\cite{ah2007}, now perhaps better known as the complement of the \emph{enhanced
power graph} of the group, following~\cite{aetal}. So we have the possibility
of proving results for many different types of algebraic structures by
considering the zero-divisor graphs of posets. This paper is intended as an
illustration of this principle.
	
Recently, Angsuman Das \cite{das, das2} defined and studied the  nonzero component graph (nonzero component union graph) of a finite dimensional vector space. Let $\mathbb{V}$ be a vector space over field $\mathbb{F}$ with  $\mathcal{B}=\{v_1,\dots,v_n\}$ as a basis and $0$ as the null vector. Then any vector $a\in \mathbb{V}$ can be uniquely expressed in linear combination of the form $a=a_1v_1+\dots+a_nv_n$. We denote this representation as a basic representation of $a$ with respect to $\{v_1,\dots,v_n\}$. Define the skeleton of $a$ with respect to $\mathcal{B}$, as
\[S_{\mathcal{B}}(a)=\{v_i\mid a_i\neq 0,  a=a_1v_1+\dots+a_nv_n\}.\] 

 Angsuman Das \cite{das} defined the \textit{nonzero component  graph} $\mathbb{IG(V)}$ with respect to $\mathcal{B}$ as follows: The vertex set of graph  $\mathbb{IG(V)}$ is $\mathbb{V}\setminus \{0\}$ and for any $a,b\in \mathbb{V}\setminus \{0\}$, $a$ is adjacent to  $b$ if and only if $a$ and $b$ share at least one $v_i$ with non-zero coefficient in their basic representation, that is, $a$ and $b$ are adjacent in $\mathbb{IG(V)}$ if and only if $S_{\mathcal{B}}(a)\cap S_{\mathcal{B}}(b)\neq\emptyset$.   It is easy to observe that the nonzero component  graph $\mathbb{IG(V)}$ with respect to a basis $\mathcal{B}$ and  the nonzero component  graph $\mathbb{IG(V)'}$ with respect to a basis $\mathcal{B'}$ are isomorphic. 

In  another paper \cite{das2}, Angsuman Das defined the \textit{nonzero component union  graph} $\mathbb{UG(V)}$ with respect to $\mathcal{B}$ as follows: The vertex set of graph  $\mathbb{UG(V)}$ is $\mathbb{V}\setminus \{0\}$ and for any $a,b\in \mathbb{V}\setminus \{0\}$, $a$ is adjacent to  $b$ if and only if $S_{\mathcal{B}}(a)\cup S_{\mathcal{B}}(b)=\mathcal{B}$. It is easy to observe that the nonzero component union graph $\mathbb{UG(V)}$ and $\mathbb{UG(V)'}$ with respect to basis $\mathcal{B}$ and $\mathcal{B'}$, respectively,  are isomorphic. 	This graph is also known as the \textit{skeleton graph}; see \cite{as2019}. 

Thus, the nonzero component graph will mean  the skeleton intersection graph and  nonzero component union graph will mean  the skeleton union graph. These names are justified by the adjacency of these two graphs.
	
Many results of  	nonzero component graphs and nonzero component union graphs of  vector spaces are similar to as that of zero-divisor graphs of ordered sets. For example, both of these graphs are weakly perfect, both of these graphs are connected  with diameter at most 2 etc. This motivated us to study these graphs from the perspective of the zero-divisor graph of an ordered sets and its join with a complete graph. By the \textit{join} of two graphs $G$ and $H$, we mean  a graph formed from disjoint copies of $G$ and $H$ by connecting each vertex of $G$ to each vertex of $H$. We denote the join of graphs 	$G$ and $H$ by $G\vee H$; see West \cite{west}. 
	
In this paper, we prove the following main result.
	\begin{thm}
		Let $\mathbb{V}$ be a $n$-dimensional  vector space over a field $\mathbb{F}$ and let $t=|V_{12\dots n}|=(|\mathbb{F}|-1)^n$. Let $\mathbb{IG(V)}$ and $\mathbb{UG(V)}$ be the nonzero component graph and nonzero component union graph respectively. 
		\begin{enumerate}
			\item $\mathbb{IG(V)}=G^c(\mathbb{L})\vee K_t$;
			\item $\mathbb{UG(V)}=G(\mathbb{L}^\partial)\vee K_t$;
		\end{enumerate}
where $\mathbb{L}$ is a lattice constructed from $\mathbb{V}$,  $\mathbb{L^\partial}$ is the dual lattice of $\mathbb{L}$ and $G(\mathbb{L})$ is the zero-divisor graph of $\mathbb{L}$.	
\end{thm}

In the last section, we  apply results on chordal and perfect zero-divisor graphs of ordered sets to decide when the graphs $\mathbb{IG(V)}$ and $\mathbb{UG(V)}$ are perfect or chordal.
	
	\section{Preliminaries}

	\par We begin with the following necessary definitions and terminologies given in Devhare et al. \cite{djl} and  Khandekar and Joshi \cite{nkvj2}. 

\begin{definition}[\cite{djl} and \cite{nkvj2}]\label{defn2.1}
	\par Let $P$ be a poset. Given any $\emptyset\neq A\subseteq P$, the \textit{upper cone} of $A$ is given by 
	$A^u=\{b\in P$ $|$ $b\geq a$ for every $a\in A\}$.  If $a\in P$, then the sets $\{a\}^u$ and
	$\{a\}^\ell$ will be denoted by $a^u$ and $a^\ell$, respectively. By $A^{u\ell}$, we mean $\{A^u\}^\ell$. Dually, we have the notion of $A^{\ell u}$.\end{definition}

	A poset $P$ with  0 is called \textit{$0$-distributive} if  $\{a,b\}^\ell=\{0\}=\{a,c\}^\ell$ implies $\{a,\{b,c\}^u\}^\ell=0$; see \cite{jw1}. Note that if $\{b,c\}^u=\emptyset$, then $\{b,c\}^{u\ell}=P$. 
	A lattice $L$ with $0$ is said to be {\it $0$-distributive }if $a\wedge b=0$ and $a\wedge c=0$ implies $a\wedge(b\vee c)=0$. Hence it is clear that if a lattice $L$ is 0-distributive, then $L$, as a poset, is a 0-distributive poset. Dually, a  lattice $L$ with $1$ is said to be \textit{$1$-distributive} if $a\vee b=1$ and $a\vee c=1$ implies $a\vee(b\wedge c)=1$.

	%\par Let $P$ be a poset with $0$. Define a \emph{zero-divisor} of $P$ to be any element of the set \linebreak $Z(P)=\{a\in P$ $|$ there exists $b\in P\setminus\{0\}$ such that $\{a,b\}^\ell=\{0\}\}$. As in \cite{LW}, the \emph{zero-divisor graph} of $P$ is the graph $G(P)$ whose vertices are the elements of $Z(P)\setminus\{0\}$ such that two vertices $a$ and $b$ are adjacent if and only if $\{a,b\}^\ell=\{0\}$. If $Z(P)\neq\{0\}$, then clearly $G(P)$ has at least two vertices, and $G(P)$ is connected with diameter at most three (\cite[Proposition 2.1]{LW}). %We abuse the notation  $G^*(P)$ for the \emph{zero-divisor graph} of $P$ with the vertex set  $P\setminus\{0,1\}$, if $P$ has the greatest element 1, if $P$ do not have 1, then the vertex set is $P\setminus\{0\}$. Further, two vertices $a$ and $b$ are adjacent in $G^*(P)$ if and only if $\{a,b\}^\ell=\{0\}$. So from the notation $G(P)$ or $G^*(P)$, the underlined vertex set of the zero-divisor graph is clear and the adjacency relation remains same in both the graphs.
	\par We set    $\mathcal{D}=P\setminus Z(P)$. The elements  $d\in\mathcal{D}$ are the {\it dense elements} of $P$.

	\par  Throughout, $P$ denotes a poset with $0$ and $q_i$, $i \in\{1,2, \cdots, n\}$ are the atoms of $P$, where $n \geq 2$. 
	
Afkhami et al. \cite{afkhami} partitioned the set $P\setminus \{0\}$ as follows.
	
Let $1\leq i_1< i_2<\dots<i_k\leq n$, where $k>0$. The notation $P_{i_1i_2\dots i_k}$ stands for the set
$$P_{i_1i_2\dots i_k}=\Bigg\{x\in P~\mathbin{\Big|}~x\in\biggl( \bigcap\limits_{s=1}^k\{q_{_{i_s}}\}^u\biggr) \mathbin{\Big\backslash}\biggl(\bigcup\limits_{j\neq i_1,i_2,\dots,i_k}\{q_{_j}\}^u\biggr)\Bigg\}. \hfill{     \hspace{.2in}  -------(\circledcirc)}$$ 
Thus $P_{i_1i_2\dots i_k}$ is the set of elements $x$ of $P$ such that the atoms below $x$ are precisely $q_{i_1},q_{i_2},\ldots,q_{i_k}$
	
In \cite{afkhami}, the following observations are proved; these show that the
sets just defined partition $P\setminus\{0\}$.

\begin{enumerate}
\item If the index sets $\{i_1,\dots,i_k\}$ and $\{j_1,\dots,j_{k'}\}$ of $P_{i_1i_2\dots i_k}$ and $P_{j_1j_2\dots j_{k'}}$, respectively, are distinct, that is, $\{i_1,\dots,i_k\}\neq \{j_1,\dots,j_{k'}\}$, then $(P_{i_1i_2\dots i_k})\cap (P_{j_1j_2\dots j_k'})=\emptyset$.
\item $\displaystyle P\backslash \{0\}= \bigcupdot\limits_{\substack{k=1,\\ 1\leq i_1< i_2<\dots<i_k\leq n}}^{n} P_{i_1i_2\dots i_k}$.
\end{enumerate}
	
Define a relation $\approx$ on $P\setminus\{0\}$ as follows:
$x\approx y$ if and only if $x,y\in P_{i_1i_2\dots i_k}$ for some part $P_{i_1i_2\dots i_k}$ of the partition just defined. Thus $x\approx y$ if $x$ and $y$
are above the same atoms of $P$.

The set of equivalence classes under $\approx$ of $P\setminus \{0\}$ will be denoted by
$$[P]'=\left\{P_{i_1i_2\dots i_k}\mid\{i_1,i_2\dots,i_k\}\subseteq \{1,2,\dots,n\}, \text{and } P_{i_1i_2\dots i_k}\neq\emptyset\right\}.$$ 
Now, we  set $[P]= [P]'\cup P_0$, where $P_0=\{0\}$. We define relation $\leq$ on $[P]$ as follows. $P_{i_1i_2\dots i_k}\leq P_{j_1j_2\dots j_m}$  if and only if $b^\perp\subseteq a^\perp$, for some $a\in P_{i_1i_2\dots i_k}$ and for some $ b\in P_{j_1j_2\dots j_m}$, where $P_{i_1i_2\dots i_k}, P_{j_1j_2\dots j_m}\in [P]' $  and $P_0\leq P_{i_1i_2\dots i_k}$ for all $\{i_1,i_2\dots,i_k\}\subseteq \{1,2,\dots,n\}$. It is not very difficult to prove that $([P], \leq)$ is a poset. The least element of  $([P], \leq)$ is $P_0$ and if $P$ has the greatest element 1, then the greatest element of the poset $([P],\; \leq)$ is $P_{12\dots n}$.

	The following statements (1)--(4) are essentially proved in \cite{djl} (see Lemma 4.2, Lemma 4.5). 
	 We write these statements in terms of $P_{i_1i_2\dots i_k}$ (see \cite{nkvj2}). These properties will be used  in the sequel.
	
	\begin{lemma}\label{property} 
		The following statements are true.
		\begin{enumerate}
			\item If $q_{_1},q_{_2},\dots,q_{_n}$ are distinct atoms of ${ P}$, then $[ {q}_{_1} ],\dots,[ {q}_{_n} ]$ are distinct atoms of $[{ P}]$. Note that $[q_{_i}]=P_{i}$, for every $i\in\{1,\dots,n\}$ .
			\item If $ {a}\leq {b}$ in ${ P}$, then $[{a} ]\leq [{b} ]$ in $[{ P}]$. Moreover, $P_{i_1i_2\dots i_k}\leq P_{j_1j_2\dots j_m}$ in $[P]$ if and only if  $\{i_1,i_2,\dots, i_k\}\subseteq \{j_1,j_2,\dots, j_m\}$.

			\item $\{{a,b} \}^\ell=\{{0} \}$ in ${ P}$ if and only if $\{[{a} ],[{b} ]\}^\ell=\{[{0} ]\}$ in $[{ P}]$. Note that the lower cones are taken in the respective posets. $P_{i_1i_2\dots i_k}$ and $P_{j_1j_2\dots j_m}$ are adjacent in $G([P])$ if and only if $\{i_1,i_2,\dots, i_k\}\cap \{j_1,j_2,\dots, j_m\}=\emptyset$. Further, $ {a} \in V(G({ P}))$ if and only if $[{a} ] \in V(G([{ P}]))$.

			\item Let $[{a} ]\in V(G([{ P}]))$. Then for any ${x,y}\in [{a} ]$, $\{{x,y} \}^\ell\neq \{{0} \}$ in ${ P}$. Hence vertices of $[{a} ]$ forms an independent set in $G({ P})$. Further, if  $\{[{a} ],[{b} ]\}^\ell=\{[{0} ]\}$ in $[{ P}]$, then for any ${x}\in [{a} ]$ and for any  $ {y}\in [{b} ]$,    $\{{x,y} \}^\ell=\{{0} \}$ in ${ P}$. In particular, $[{a} ]$ and $ [{b}]$ are adjacent in $G({ [P]})$ with $|[{a} ]|=m$, $|[{b} ]|=n$, then the vertices of $[{a} ]$ and $[{b} ]$ forms an induced complete bipartite subgraph $K_{m,n}$ of $G({ P})$. Moreover, for any $x, y \in [a]$, deg$_{G(P)}(x)= $ deg$_{G(P)}(y)$.  
		 
		\end{enumerate}
	\end{lemma}

\section{Interplay}	

	 With this preparation, we now give a relation between the  skeleton intersection (union) graph  of a finite dimensional vector space and the zero-divisor graph of a poset.
	 
	 For this, let $\mathbb{V}$ be a finite dimensional vector space over a field $\mathbb{F}$ with  $\mathcal{B}=\{v_1,\dots,v_n\}$ as a basis. Let $1\leq i_{1}<i_{2}\dots <i_{k}\leq n.$ The notation, $V_{\emptyset}=\{0\}$ (null vector/zero vector) and  $V_{i_{1}i_{2}\dots i_{k}}$ stands for the set 
\[V_{i_{1}i_{2}\dots i_{k}}=\left\{a\in\mathbb{V}\mid (a_j\neq0)
\Leftrightarrow(j\in\{i_1,\dots,i_k\})\right\},\]
where $a=a_1v_1+\dots+a_nv_n$. We also denote this set by $V_I$, where
$I=\{i_1,i_2,\ldots,i_k\}$.
	
For any $\{i_1,\dots,i_k\}\subseteq\{1,2,\dots,n\}$, we have $v_{i_1}+\dots+v_{i_k}=w$ (say). Clearly, $w\in V_{i_{1}i_{2}\dots i_{k}}$. Thus $V_{i_{1}i_{2}\dots i_{k}}\neq\emptyset$ for every $\{i_1,\dots,i_k\}\subseteq\{1,2,\dots,n\}$. Also observe that $|V_{i_{1}i_{2}\dots i_{k}}|=(|\mathbb{F}|-1)^k$ for all nonempty set $\{i_1,\dots,i_k\}\subseteq\{1,2,\dots,n\}$. 
	
Since any vector $a\in \mathbb{V}$ can be uniquely expressed as $a=a_1v_1+\dots+a_nv_n$, we have $V_{i_{1}\dots i_{k}}\cap V_{j_{1}\dots j_{l}}=\emptyset$ if $\{i_1,\dots,i_k\}$ and $\{j_1,\dots,j_l\}$ are distinct subsets of $\{1,2,\dots,n\}$.
	
For any $a\in \mathbb{V}$, if $a=0$, then $0=a\in V_{\emptyset}$; otherwise, if $a\neq 0$, then $a\in V_{i_{1}\dots i_{k}}$ for some nonempty subset $\{i_1,\dots,i_k\}$ of $\{1,2,\dots,n\}$. 
	
Thus, $\mathbb{V}=\bigcupdot\limits_{I\subseteq\{1,\dots,n\}}V_I$.
	
We define $[\mathbb{V}]=\big\{V_I|I\subseteq\{1,\dots,n\}\big\}$. Define a relation $\leq$ on $[\mathbb{V}]$ as: $V_I\leq V_J$ if and only if $I\subseteq J$. Clearly, $[\mathbb{V}]$ is a poset. In fact, $[\mathbb{V}]$ is a lattice under $V_I \wedge V_J=V_{I\cap J}$ (meet) and $V_I \vee V_J=V_{I\cup J}$ (join). Further, we observe that $[\mathbb{V}]$ contains $2^n$ elements and is a Boolean lattice isomorphic to $P(X)$ for $X=\{1,\ldots,n\}$. In fact, the map $I\mapsto V_I$ for $I\subseteq X$, is an isomorphism between  $[\mathbb{V}]$ and $P(X)$.
	
Now, we construct a lattice $\mathbb{L}$ from $[\mathbb{V}]$ such that the zero-divisor graph of $\mathbb{L}$ is related to the  skeleton intersection graph and  skeleton  union graph of  $\mathbb{V}$.
	
 We replace $V_I\in [\mathbb{V}]$  by the chain of elements of $V_I$ in $\mathbb{L}$ in some pre-determined well-order, where $I\neq \emptyset$ and $I\neq \{1,\dots,n\}$. The elements $V_{\emptyset}$ and $V_{\{1,\dots,n\}}$ of $[\mathbb{V}]$ are replaced  by $0$ and $1$ respectively, in $\mathbb{L}$.
	
We illustrate this construction with the following example.
	
Consider $\mathbb{V}$ as a $3-$dimensional vector space over field $\mathbb{F}$, where $|\mathbb{F}|=3$. Let $\mathcal{B}=\{v_1,v_2,v_3\}$ be a basis of  $\mathbb{V}$ and let $\mathbb{F}=\{\overline{0},\overline{1},\overline{2}\}$.
	
\begin{figure}[htbp]
\begin{center}
\begin{tikzpicture}[scale =0.9]
			
\draw [fill=black] (0,0) circle (.1); \draw [fill=black] (-2,2) circle (.1);\draw [fill=black] (0,2) circle (.1);\draw [fill=black] (2,2) circle (.1);\draw [fill=black] (-2,5) circle (.1);\draw [fill=black] (0,5) circle (.1);\draw [fill=black] (2,5) circle (.1);\draw [fill=black] (0,7) circle (.1);

\draw (0,0) -- (-2,2)--(-2,5)--(0,7)--(0,5)--(2,2)--(0,0)--(0,2)--(2,5)--(0,7)--(0,5)--(-2,2);
\draw (2,2)--(2,5)--(0,2)--(-2,5);			
\node [below] at (0,-0.1) {$V_{\emptyset}$};\node [left] at (-2,2) {$V_1$};\node [left] at (0,2) {$V_2$};	\node [right] at (2,2) {$V_3$};\node [left] at (-2,5) {$V_{12}$};\node [left] at (0,5) {$V_{13}$};\node [right] at (2,5) {$V_{23}$};\node [above] at (0,7) {$V_{123}$};		  
			
\node [below] at (0,-0.7) {(A) \text{The lattice} $[\mathbb{V}]$};

\begin{scope}[shift={(7,0)}]

\draw [fill=black] (0,0) circle (.1); \draw [fill=black] (-2,1.5) circle (.1);\draw [fill=black] (-2,2) circle (.1);\draw [fill=black] (-2,4) circle (.1);\draw [fill=black] (-2,4.5) circle (.1);\draw [fill=black] (-2,5) circle (.1);\draw [fill=black] (-2,5.5) circle (.1);\draw [fill=black] (0,7) circle (.1);\draw [fill=black] (0,5.5) circle (.1);\draw [fill=black] (0,5) circle (.1);\draw [fill=black] (0,4.5) circle (.1);\draw [fill=black] (0,4) circle (.1);\draw [fill=black] (0,2) circle (.1);\draw [fill=black] (0,1.5) circle (.1);\draw [fill=black] (2,1.5) circle (.1);\draw [fill=black] (2,2) circle (.1);\draw [fill=black] (2,4) circle (.1);\draw [fill=black] (2,4.5) circle (.1);\draw [fill=black] (2,5) circle (.1);\draw [fill=black] (2,5.5) circle (.1);
		
\draw (0,0)--(-2,1.5)--(-2,5.5)--(0,7)--(0,4)--(2,2)--(2,1.5)--(0,0)--(0,2)--(2,4)--(2,5.5)--(0,7)--(0,4)--(-2,2);
\draw (2,2)--(2,4)--(0,2)--(-2,4);		
\node [below] at (0,-0.1) {$0$}; \node [left] at (-2,1.5) {$v_1$};\node [left] at (-2,2) {$\overline{2}v_1$};
\node [left] at (0,1.5) {$\overline{2}v_2$};\node [left] at (0,2) {$v_2$};\node [right] at (2,1.5) {$v_3$};\node [right] at (2,2) {$\overline{2}v_3$};

\node [left] at (-2,4) {$v_1+v_2$};\node [left] at (-2,4.5) {$v_1+\overline{2}v_2$};\node [left] at (-2,5) {$\overline{2}v_1+v_2$};\node [left] at (-2,5.5) {$\overline{2}v_1+\overline{2}v_2$};

\node [left] at (0,5.5) {$v_1+v_3$};\node [left] at (0,4) {$v_1+\overline{2}v_3$};\node [left] at (0,5) {$\overline{2}v_1+v_3$};\node [left] at (0,4.5) {$\overline{2}v_1+\overline{2}v_3$};

\node [right] at (2,4.5) {$v_2+v_3$};\node [right] at (2,5.5) {$v_2+\overline{2}v_3$};\node [right] at (2,4) {$\overline{2}v_2+v_3$};\node [right] at (2,5) {$\overline{2}v_2+\overline{2}v_3$};

\node [above] at (0,7) {$1$};		
	
\node [below] at (0,-0.7) {(B) \text{The lattice} $\mathbb{L}$ derived from $[\mathbb{V}]$};	
			
\end{scope}
\end{tikzpicture}
\end{center}
\caption{}\label{v3f3}
\end{figure}

To derive the properties of $\mathbb{L}$, we need the following results which are straight forward. Hence proofs are omitted.

 \begin{lemma}\label{ll'}
 	Let $L$ be a finite lattice and $L'$ be a poset obtained from $L$ by replacing an element of $L$ by a bounded chain. Then the following statements hold:
 	\begin{enumerate}
 		\item The poset $L'$ is a  lattice.
 		\item If $L$ is a $0$-distributive ($1$-distributive) lattice, then $L'$ is   a $0$-distributive ($1$-distributive) lattice.
 		
 	\end{enumerate}
 \end{lemma}	
	
	Since $[\mathbb{V}]$ is a Boolean lattice and every Boolean lattice is $0$-distributive and $1$-distributive, so is $\mathbb{L}$ by Lemma \ref{ll'}. 
	\begin{lemma}
		The lattice $\mathbb{L}$ derived from $[\mathbb{V}]$ is a $0$-distributive as well as $1$-distributive.
	\end{lemma}

Now, we construct the poset $[\mathbb{L}]$ under the equivalence relation $\approx$ on $\mathbb{L}$.	It is easy to observe that $V_\emptyset=\mathbb{L}_0$ (the set $P_0$ mentioned in Definition \ref{defn2.1}) and $V_I=\mathbb{L}_I$ for nonempty subset $I\subseteq \{1,\dots,n\}$. Thus, $[\mathbb{V}]\cong [\mathbb{L}]$. 

Henceforth we are using $[\mathbb{V}]$ and $[\mathbb{L}]$, $V_I$ and $\mathbb{L}_I$ without any distinction. Also, we are using Lemma \ref{property} in terms of $V_I$. 
	Note that for any $x,y\in \mathbb{L}$ $(x\in V_I$ and $y\in V_J$ for some $I,J\subseteq\{1,\dots,n\})$.
	
	\begin{enumerate}
		\item $x\wedge y=0$ in $\mathbb{L}$ if and only if $V_I\wedge V_J=V_{\emptyset}$ if and only if $I\cap J=\emptyset$.
		
		\item $x\vee y=1$ in $\mathbb{L}$ if and only if $V_I\vee V_J=V_{\{1,\dots,n\}}$ if and only if $I\cup J=\{1,\dots,n\}$.
	\end{enumerate}
	
	Now, we make a very important observation. Recall that $\mathcal{B}$ is
a basis for $\mathbb{V}$. For any $a\in \mathbb{V}$, 
	
	\begin{enumerate}
		\item If $a=0$, then $S_{\mathcal{B}}(a)=\emptyset$ and $a\in V_{\emptyset}$
		
		\item If $a\neq 0$, then $S_{\mathcal{B}}(a)\subseteq \mathcal{B}$. In particular, $S_{\mathcal{B}}(a)=\{v_{i_1},\dots,v_{i_k}\}\subseteq \mathcal{B}$ if and only if $a\in V_I$, where $I=\{i_1,\dots,i_k\}$.
	\end{enumerate}

	Note that $L^\partial$ denotes the dual of a lattice $L$ (that is $a\leq b$ in $L$ if and only if $b\leq a$ in $L^\partial$), which is also lattice. If $L$ is a $0-1$-distributive lattice, then so is $L^\partial$.
	
	Let $G(\mathbb{L})$ be the zero-divisor graph of the lattice $\mathbb{L}$ and $G^c(\mathbb{L})$ be its graph complement.
	With this preparation, we are ready to prove our  two main theorems.
	
	\begin{thm}\label{igv}
		$\mathbb{IG(V)}=G^c(\mathbb{L})\vee K_t$, where $t=|V_{12\dots n}|=(|\mathbb{F}|-1)^n$.
	\end{thm}
	
	\begin{proof}
		Note that $V(G(\mathbb{L}))=\mathbb{L}\setminus\{0,1\}$, as $[\mathbb{V}]$ is Boolean. Further, $V(\mathbb{IG(V)})= V(G(\mathbb{L}))\cup V_{12\dots n}$. For $u,v\in \mathbb{V}\setminus \{\{0\} \cup V_{12\dots n}\}$. The vertices $u$ and $v$ are adjacent in  $G(\mathbb{L})$ if and only if $u\wedge v=0$ in $\mathbb{L}$ if and only if $V_I\wedge V_J=V_{\emptyset}$, where $u\in V_I, v\in V_J$ if and only if $I\cap J=\emptyset$ if and only if $S_{\mathcal{B}}(u)\cap S_{\mathcal{B}}(v)=\emptyset$ if and only if $u$ and $v$ are not adjacent in $\mathbb{IG(V)}$. This proves that the induced subgraph of $\mathbb{IG(V)}$ on the vertex set $\mathbb{V}\setminus \{\{0\} \cup V_{12\dots n}\}$ is equal to $G^c(\mathbb{L})$.
		
		 Let $w\in V_{12\dots n}$ and $(0\neq)u\in \mathbb{V}$ such that $u\neq w$, then $S_{\mathcal{B}}(w)=\{v_1,\dots,v_n\}$ and $S_{\mathcal{B}}(u)=\{v_{i_1},\dots,v_{i_k}\}$ for some $\{i_1,\dots,i_k\}\subseteq \{1,\dots,n\}$ and $k\geq 1$. Clearly, $S_{\mathcal{B}}(w)\cap S_{\mathcal{B}}(u)\neq \emptyset$. This implies that  any $w\in V_{12\dots n}$ is adjacent to a nonzero $ u\in \mathbb{V}\setminus\{w\}$ in $\mathbb{IG(V)}$. This proves that $\mathbb{IG(V)}=G^c(\mathbb{L})\vee K_t$, where $t=|V_{12\dots n}|=(|\mathbb{F}|-1)^n$.\end{proof}

	\begin{thm}\label{ugv}
		$\mathbb{UG(V)}=G(\mathbb{L}^\partial)\vee K_t$, where $t=|V_{12\dots n}|=(|\mathbb{F}|-1)^n$.
	\end{thm}
	
	\begin{proof}
	Consider the induced subgraph of $\mathbb{UG(V)}$ on the set $\mathbb{V}\setminus \{\{0\} \cup V_{12\dots n}\}$. Note that  the greatest element $1$ of $\mathbb{L}$ is the zero element of $\mathbb{L}^\partial$. Let $u,v\in \mathbb{V}\setminus \{\{0\} \cup V_{12\dots n}\}$. The vertices $u$ and $v$ are adjacent in  $G(\mathbb{L}^\partial)$ if and only if $u\wedge v=1$ in $\mathbb{L}^\partial$ if and only if $u\vee v=1$ in $\mathbb{L}$ if and only if $V_I\vee V_J=V_{12\dots n}$, where $u\in V_I, v\in V_J$ if and only if $I\cup J=\{1,2,\dots,n\}$ if and only if $S_{\mathcal{B}}(u)\cup S_{\mathcal{B}}(v)=\mathcal{B}$ if and only if $u$ and $v$ are  adjacent in $\mathbb{UG(V)}$. This proves that the induced subgraph of $\mathbb{IG(V)}$ on the vertex set $\mathbb{V}\setminus \{\{0\} \cup V_{12\dots n}\}$ is equal to $G(\mathbb{L}^\partial)$.
		
	 Let $w\in V_{12\dots n}$ and $(0\neq)u\in \mathbb{V}$ such that $u\neq w$. Then $S_{\mathcal{B}}(w)=\{v_1,\dots,v_n\}$ and $S_{\mathcal{B}}(u)=\{v_{i_1},\dots,v_{i_k}\}$ for some $\{i_1,\dots,i_k\}\subseteq \{1,\dots,n\}$ and $k\geq 1$. Clearly, $S_{\mathcal{B}}(w)\cup S_{\mathcal{B}}(u)=\{v_1,\dots,v_n\}$. This implies that any $w\in V_{12\dots n}$, is adjacent to any nonzero $  u\in \mathbb{V}\setminus\{w\}$ in $\mathbb{UG(V)}$. This proves that $\mathbb{UG(V)}=G(\mathbb{L}^\partial)\vee K_t$, where $t=|V_{12\dots n}|=(|\mathbb{F}|-1)^n$. \end{proof}
		
From the above two results, it is clear that nonzero component union graph alias the skeleton union graph as well as nonzero component  graph alias skeleton intersection graph can be studied using the zero-divisor graph of ordered sets; in particular, zero-divisor graph of 0-1 distributive lattices.
		
Let $\mathbb{V}$ be a $n$-dimensional vector space over a field $\mathbb{F}$ and let $\mathcal{B}=\{v_1,\dots,v_n\}$ be a basis of $\mathbb{V}$. Then the map
\[a=a_1v_1+\dots+a_nv_n \mapsto (a_1,\dots,a_n)\]
is a vector spacce isomorphism from	$\mathbb{V}$ to $\mathbb{F}^n$, where $\mathbb{F}^n=\mathbb{F}\times \dots \times \mathbb{F}$ ($n$-times). 
Now $\mathbb{F}^n$ is also a ring (the direct product of $n$ copies of
$\mathcal{F}$. We prove that skeleton intersection  graph of $\mathbb{V}$ is the graph join of the complement of the ring-theoretic zero-divisor graph of  $\mathbb{F}^n$ and a complete graph. In fact, we prove the following result.
		
\begin{thm}
Let $\mathbb{V}$ be a $n$-dimensional vector space over a field $\mathbb{F}$ and let $\mathcal{B}=\{v_1,\dots,v_n\}$ be a basis of $\mathbb{V}$. Then 	$\mathbb{IG(V)}\cong \Gamma^c(\mathbb{F}^n)\vee K_t$, where $t=|V_{12\dots n}|=(|\mathbb{F}|-1)^n$.
\end{thm}		

\begin{proof} 	Let $\mathbb{V}$ be a $n$-dimensional vector space over a field $\mathbb{F}$ and let $\mathcal{B}=\{v_1,\dots,v_n\}$ be a basis of $\mathbb{V}$. Consider the ring $\mathbb{F}^n$. The set of all nonzero zero-divisors of $\mathbb{F}^n$ is $\mathbb{F}^n\setminus \big\{(0,\dots,0) \cup U(\mathbb{F}^n)\big\}$, where $U(\mathbb{F}^n)$ is the set of units of $\mathbb{F}^n$.  Now an element $a\in\mathbb{V}$ satisfies $a\in V_{12\dots n}$ if and only if $(a_1,\dots,a_n)\in U(\mathbb{F}^n)$: for if $a_i\ne0$ for all $i\in\{1,\ldots,n\}$ then
$a^{-1}=b=(a_1^{-1},a_2^{-1},\ldots,a_n^{-1})$, whereas if $a_i=0$ for some $i$ then $a$ has no inverse in $\mathbb{F}^n$. Therefore we have a bijective map from $V(\mathbb{IG(V)})$ to $V(\Gamma(\mathbb{F}^n))\cup U(\mathbb{F}^n)$ and vice-versa.

 For $u,v\in \mathbb{V}\setminus \{\{0\} \cup V_{12\dots n}\}$. The vertices $u$ and $w$ are adjacent in  $\mathbb{IG(V)}$ if and only if  $S_{\mathcal{B}}(u)\cap S_{\mathcal{B}}(w)\neq \emptyset$ if and only if 
 $u_i\neq 0$ and $w_i\neq 0$
 for some $1\leq i\leq n$ if and only if $(u_1,\dots,u_n)$ and $(w_1,\dots,w_n)$ are not adjacent in $\Gamma(\mathbb{F}^n)$. This proves that the induced subgraph of $\mathbb{IG(V)}$ on the vertex set $\mathbb{V}\setminus \{\{0\} \cup V_{12\dots n}\}$ is isomorphic to $\Gamma^c(\mathbb{F}^n)$.
	
From the proof of Theorem \ref{igv}, any $w\in V_{12\dots n}$ is adjacent to $u$ for  a nonzero $  u\in \mathbb{V}\setminus\{w\}$ in $\mathbb{UG(V)}$.  This proves that $\mathbb{IG(V)}\cong \Gamma^c(\mathbb{F}^n)\vee K_t$, where $t=|V_{12\dots n}|=(|\mathbb{F}|-1)^n$.	
\end{proof}

We will close this section by observing that the reduced graph of $\mathbb{IG(V)}$ ($\mathbb{UG(V)}$) is  the graph join of the zero-divisor graph of a Boolean ring (equivalently, Boolean lattice) and the complete graph. For this, we need the following results. It is well-known that the compressed lattice of a finite 0-distributive lattice under the equivalence relation $\approx$ (see Definition \ref{defn2.1}) is always a Boolean lattice; see \cite[Lemma 4.6]{djl}. In the next result,  the finiteness of set of atoms is enough to prove that $[L]$ is Boolean, if $L$ is a 0-distributive lattice.

	\begin{theorem}[{\cite[Theorem 2.15]{nkvj2}}]\label{nboolean}
		Let $L$ be a $0$-distributive bounded lattice with finitely many atoms. Then $[L]$ is a Boolean lattice.
	\end{theorem}		

Let $G$ be a finite graph. The set $\{u\in V(G)~~|~~u-v\in E(G)\}$ be the \textit{neighborhood} of a vertex $v$ in a graph $G$, denoted by $N_G(v)$. If there is no ambiguity about the graph under consideration, then we simply write $N(v)$.    Consider  a relation (given in \cite{nkvj2})  on $G$ such that $u\simeq v$ if and only if either $u=v$, or  $u-v\in E(G)$ with $N(u)\setminus \{v\}= N(v)\setminus \{u\}$. Clearly, $\simeq$ is an equivalence relation on $V(G)$. The equivalence class of $v$ is the set $\{u\in V(G)~~|~~u\simeq v\}$, denoted by $[v]^\simeq$. Denote the set $\{[v]^\simeq~~|~~ v\in V(G)\}$ by $G_{\red}$. Define $[u]-[v]$ is an edge in $E(G_{\red})$ if and only if $u-v\in E(G)$, where $[u]\neq [v]$.

\begin{remark}\label{gred}
	It is easy to observe that, if $G^c(P)$ is  the complement of the zero-divisor graph $G(P)$, then $(G^c(P))_{\red}=G^c([P])$. 
	
\end{remark}

Bagheri et al. \cite{bagheri} considered the following relation on a graph $G$: $u\approxeq v$ if and only if $N_G(u)= N_G(v)$. Clearly, $\approxeq$ is an equivalence relation on $V(G)$. The equivalence class of $v$ is the set $\{u\in V(G)~~|~~u\approx v\}$, denoted by $[v^\approxeq]$. Denote the set $\{[v^\approxeq]~~|~~ v\in V(G)\}$ by $[V(G)]$. Define $[u^\approxeq]-[v^\approxeq]$ is an edge in $E([G])$ if and only if $u-v\in E(G)$, where $[u^\approxeq]\neq [v^\approxeq]$. Let $[G]^{\approxeq}$ be a simple graph whose vertex set is $[V(G)]$, and edge set is $E([G])$. 

\begin{remark}\label{g[]}
	It is easy to observe that, if $G(P)$ be  the  zero-divisor graph, then $[G(P)]=G([P])$.	
\end{remark}

Since $\mathbb{L}$ is a $0-1$-distributive lattice with $n$ atoms, then $[\mathbb{L}]$	and $[\mathbb{L^\partial}]$ both are Boolean lattices having $2^n$ elements. This implies  that $G([\mathbb{L}])\cong G([\mathbb{L}^\partial])\cong \Gamma(\mathbb{Z}_2^n)$. In view of Theorem \ref{igv}, \ref{ugv} and Remark \ref{gred}, \ref{g[]}, we have the following result.

\begin{corollary}
	\begin{enumerate}
	\item 	$(\mathbb{IG(V)})_{\red}=\Gamma^c(\mathbb{Z}_2^n)\vee K_1$
	
	\item $[\mathbb{UG(V)}]^{\approxeq}=\Gamma(\mathbb{Z}_2^n)\vee K_t$, where $t=|V_{12\dots n}|=(|\mathbb{F}|-1)^n$.	
	\end{enumerate}
\end{corollary}

\section{Applications}
	
A \textit{chord} of a cycle $C$ of graph $G$ is an edge that is not in $C$ but has both its end vertices in $C$. A graph $G$ is \textit{chordal} if every cycle of length at least $4$ has a chord, \textit{i.e.},  $G$ is chordal if and only if it does not contain induced cycle of length at least $4$. A perfect graph is a graph in which the chromatic number of every induced subgraph equals the order of the largest clique of that subgraph (clique number). The Strong Perfect Graph Theorem~\cite{strongperfect} asserts that a graph $G$  is perfect if and only if neither $G$ nor $G^c$ contains an induced odd cycle of length at least $5$.

\begin{theorem}[{\cite[Theorem 1.1] {nkvj2}}] \label{zdgchordal}
Let $P$ be a finite poset such that $[P]$ is a Boolean lattice. Then
\begin{itemize}
\item[\textbf{(A)}] $G(P)$ is chordal if and only if one of the following holds:
	\begin{enumerate}
	\item $P$ has exactly one atom;
	\item $P$ has exactly two atoms with $|P_i|=1$ for some $i\in \{1,2\}$;
	\item $P$ has exactly three atoms with $|P_i|=1$ for all $i\in \{1,2,3\}$.
	\end{enumerate}
\item[\textbf{(B)}] $G^c(P)$ is chordal if and only if the  number of atoms of $P$ is at most $3$.	
\item[\textbf{(C)}] $G(P)$ is perfect if and only if $P$ has at most 4 atoms.
\end{itemize}
\end{theorem}

\begin{remark}\label{obs1} It is easy to observe that the following conditions
are equivalent:
\begin{enumerate}
\item $G+ I_m$ is a chordal (resp.~perfect) graph;
\item $G$ is a chordal (resp.~perfect) graph;
\item $G\vee K_m$ is a chordal (resp.~perfect) graph.
\end{enumerate}
\end{remark}		
	
In view of Theorem \ref{igv}, \ref{ugv}, \ref{zdgchordal}	and Remark \ref{obs1}, we have the following result.	

\begin{corollary}
Let $\mathbb{V}$ be a finite dimensional vector space over finite field $\mathbb{F}$. Then 
\begin{enumerate}
\item The skeleton intersection graph $\mathbb{IG(V)}$ is chordal if and only if $\dim(\mathbb{V}) \leq 3$.
\item The skeleton union graph $\mathbb{UG(V)}$ is chordal if and only if either  $\dim(\mathbb{V})=1$ or $\dim(\mathbb{V}) \in\{2, 3\}$ with $|\mathbb{F}|=2$.
\item The skeleton intersection graph $\mathbb{IG(V)}$ is perfect if and only if the skeleton union graph $\mathbb{UG(V)}$ is perfect; this occurs if and only if  $\dim(\mathbb{V}) \leq 4$.	
\end{enumerate}
\end{corollary}
		
Many other properties of skeleton intersection and skeleton union graphs can be obtained from the zero-divisor graphs of posets. Such properties include being
weakly perfect, Eulerian or Hamiltonian.
	
\noindent \textbf{Acknowledgment:} The first author is financially supported by the Council of Scientific and Industrial Research(CSIR), New Delhi via Senior Research Fellowship Award Letter No. 09/137(0620)/2019-EMR-I.
	
%\noindent\textbf{Compliance with ethical standards}
	
\noindent\textbf{Conflict of interest:} The authors declare that there is no conflict of	interests regarding the publishing of this paper.
	
\noindent\textbf{Authorship Contributions:} The authors read and approved the final version of the manuscript.

\end{document}